\documentclass[12pt,a4paper]{article}

\usepackage{amsmath,amsthm,amssymb,amscd,a4wide}

\usepackage{dsfont}
\usepackage{mathrsfs}
\usepackage{setspace}
\usepackage{hyperref}

\usepackage{xcolor}

\numberwithin{equation}{section}

\newtheorem{theorem}{Theorem}[section]
\newtheorem{lemma}[theorem]{Lemma}
\newtheorem{prop}[theorem]{Proposition}

\newtheorem{remark}[theorem]{Remark}

\theoremstyle{definition}

\begin{document}
	
	\thispagestyle{empty}
	
	\vspace*{1cm}
	
	\begin{center}
		
		{\LARGE\bf On the spectral gap of the path graph in the limit of large volume} \\

		\vspace*{2cm}
		
		{\large Joachim Kerner \footnote{E-mail address: {\tt Joachim.Kerner@fernuni-hagen.de}}}%
		
		\vspace*{5mm}
		
		Department of Mathematics and Computer Science\\
		FernUniversit\"{a}t in Hagen\\
		58084 Hagen\\
		Germany\\
		\vspace*{2cm}

		{\large Pavlo Yatsyna \footnote{E-mail address: {\tt yatsyna@karlin.mff.cuni.cz}} }%
		
		\vspace*{5mm}
		
		Faculty of Mathematics and Physics, Department of Algebra\\
		Charles University\\ 
		Sokolov\-sk\' a 83, 18600 Praha~8\\ 
		Czech Republic\\
		\vspace*{2cm}

	\end{center}
	
	\vfill
	
	\begin{abstract} In this paper we study the spectral gap of the path graph and illustrate an interesting effect which has been described recently in the continuous setting. More explicitly, in the large-volume limit and in the presence of a certain external potential, it is shown that the spectral gap converges to zero strictly faster than it does for the free Laplacian. The underlying mechanism is a combination of the increase in volume and an effective degeneracy of the ground state in the limiting regime.
	\end{abstract}
	
	\newpage
	
	\section{Introduction}
	
	In this paper we study the spectral gap of a well-known discrete graph, the so-called path graph, in a limit where the number of vertices tends to infinity. Our study is motivated by an interesting observation made recently in~\cite{KernerTaufer} in the case of one-dimensional Schrödinger operators. There, it was realized that the spectral gap of the Laplacian plus some non-negative potential of compact support closes strictly faster in the limit of large volume than the spectral gap of the Laplacian without an additional external potential. In other words, although the potential is identically zero on an increasing part of the configuration space, it still manifests itself in a different asymptotic behaviour of the spectral gap. Intuitively, this somewhat surprising effect is the result of two different contributions: On the one hand, increasing the volume of the system implies that the two lowest eigenvalues converge to zero which leads to a convergence of the spectral gap. On the other hand, the external potential is becoming a higher and higher barrier (for the particles occupying the first eigenstates, since their energies converge to zero) and therefore leads to an effective decoupling of configuration space into two (congruent) parts. Consequently, since the ground state of an operator defined over the disjoint union of two copies of the same domain is degenerate, this effective degeneracy of the ground state in the limit of large volume adds to the decay of the spectral gap. With this paper it is our aim to illustrate this effect in the context of discrete graphs, seemingly for the first time.  
	
	Finally, for related results on the spectral gap in the continuous case let us refer to the classical papers \cite{KirschGapII,KirschGap, AB,Abramovich,Lavine} and the more recent ones \cite{AK,ACH,KernerCompact}.

	\section{The model}
	
	We study the path graph $G_k=(V_k,E_k)$ with an odd number of vertices $|V_k|=2k+1$, $k \in \mathbb{N}$. We recall that all vertices of $G_k$ have exactly two neighbours except for the outer two vertices that have one neighbouring vertex only. It is convenient in the following to numerate the vertices according to the set $V_k:=\{-k,...,0,...,+k\}$. We remark that $E_k$ refers to the edge set of the graph.

	The underlying (finite-dimensional) Hilbert space is given by $\mathcal{H}=\mathbb{C}^{|V_k|}$. On $\mathcal{H}$ we then introduce the standard (weighted) discrete Laplacian operator $\mathcal{L}_{\gamma}:\mathbb{C}^{|V_k|} \rightarrow \mathbb{C}^{|V_k|}$ defined via 
	\begin{equation*}
	(\mathcal{L}_{\gamma}f)(v):=\sum_{w \in V_k: |v-w|=1}\gamma_{wv}(f(v)-f(w))\ , \quad v \in V_k\ , \quad f \in \mathcal{H}\ .
	\end{equation*} 
	Here $\gamma_{wv}=\gamma_{vw} \in (0,\infty)$ are real-valued and positive edge weights. The \textit{unweighted} path graph is obtained through setting $\gamma_{vw}= 1$ for all $v,w \in V_k$. The quadratic form associated with $\mathcal{L}_{\gamma}$ is defined by
	\begin{equation*}\label{QForm}
	q_{w}[f]:=\frac{1}{2}\sum_{i,j \in V_k}w_{ij}|f(i)-f(j)|^2\ , \qquad f \in \mathcal{H}\ ,
	\end{equation*}
	 with weights $w_{ij}$ defined by
	\begin{equation*}w_{ij}=:\begin{cases} \gamma_{ij} \quad \text{for} \quad |i-j|=1\ , \\
	0 \quad \text{else}\ .
	\end{cases}
	\end{equation*}
	Since the Laplacian $\mathcal{L}_{\gamma}$ is a self-adjoint (non-negative) operator on a finite-dimensional Hilbert space, its spectrum consists of (non-negative) eigenvalues only. Furthermore, inserting the vector $f=(1,...,1)^T$ in $q_{w}[\cdot]$ shows that the lowest eigenvalue of $\mathcal{L}_{\gamma}$ is indeed zero. 
	
	At various points in the paper we want to introduce a compactly supported and non-negative external potential which shall be located only at the zero vertex. This means that the quadratic form then reads 
	\begin{equation*}\label{QFormII}
	q_{w,u}[f]:=\frac{1}{2}\sum_{i,j \in V_k}w_{ij}|f(i)-f(j)|^2+u|f(0)|^2\ , \qquad f \in \mathcal{H}\ ,
	\end{equation*}
	with $u \geq 0$ representing the external potential at zero vertex. The associated self-adjoint operator then becomes $H_{\gamma}:=\mathcal{L}_{\gamma}+u\delta_0$ where $(\delta_0f):=f(0)$ is the $\delta$-function. Again, the operator $H_{\gamma}$ is non-negative and its spectrum consists of non-negative eigenvalues only. We write them as $\lambda_0(V_k,w,u) \leq \lambda_1(V_k,w,u) \leq ... \leq \lambda_{|V_k|-1}(V_k,w,u) $. 
	
	Finally, let us introduce the central object of interest in this paper: the spectral gap $\Gamma(V_k,w,u)$ which is defined as
	\begin{equation}
	\Gamma(V_k,w,u):=\lambda_1(V_k,w,u)-\lambda_0(V_k,w,u)\ .
	\end{equation}
	We remark that, for each value of $k \in \mathbb{N}$, $\Gamma(V_k,w,u) > 0$ since the ground state is non-degenerate (compare with Proposition~\ref{SymmetryGS}). Also, for more details on the spectral theory of graphs we refer to~\cite{Chung,Brouwer}.
	\section{Main results}
	In this section we study the spectral gap $\Gamma(V_k,w,u)$ in more detail. We start with the following statement which characterizes convergence of the spectral gap in the case of vanishing potential ($u=0$) and with weights $\gamma_{vw}=1$, $v,w \in V_k$.
	\begin{prop}\label{PropPavloI} Consider a path graph with Hamiltonian $H_{\gamma}$, $u=0$ and edge weights $\gamma_{vw}= 1$, $v,w \in V_k$. Then, 
		\begin{equation*}
	\lim_{k \rightarrow \infty}	|V_k|^2 \cdot  \Gamma(V_k,w,u)=\pi^2\ .
		\end{equation*}
		Furthermore, whenever $u > 0$ and $\gamma_{vw}=1$ for $v,w \in V_k$, one has 
		\begin{equation*}
		\lim_{k \rightarrow \infty}	|V_k|^2 \cdot  \lambda_1(V_k,w,u)=\pi^2\ .
		\end{equation*}
	\end{prop}
\begin{proof} In a first step we observe that $\lambda_0(V_k,w,u)=0$ since $u=0$. On the other hand, for the second eigenvalue one has the explicit expression~\cite{Brouwer}
	$$\lambda_1(V_k,w,u)=2-2\cos\left(\frac{\pi}{|V_k|} \right)\ . $$
From this the first statement follows immediately.

The second statement is a direct consequence of standard interlacing theorems~\cite{Brouwer}. Deleting the zero vertex, we find two copies of a graph with the smallest eigenvalue $2-2\cos\left(\frac{\pi}{|V_k|} \right)$ and which equals $\lambda_1(V_k,w,u)$.
\end{proof}
\begin{remark} Proposition~\ref{PropPavloI} characterizes the behaviour of the spectral gap of our reference operator, namely, the free (meaning zero external potential) discrete Laplacian of the unweighted path graph.
	
	Also, Proposition~\ref{PropPavloI} implies that the spectral gap always decreases when adding a potential to the zero vertex. In other words, the spectral gap does not increase! This is somewhat interesting when compared with related results in the continuum case \cite{AB,Abramovich,Lavine}. Here, the spectral gap might increase or decrease, depending on properties of the potential such as convexity. Of course, the situation might be different in the case of more complex potentials or graphs.
\end{remark}
In a next step we collect some results that characterize in more detail the ground state eigenvector $u_0=(u_0(-k),...,u_0(+k))^T \in \mathcal{H}$ of the operator $\mathcal{H}_{\gamma}$. Note that we shall always assume in this paper that $u_0$ is normalized, i.e., $\|u_0\|_{\mathcal{H}}=1$. Also, for notational simplicity, we shall not state the $k$-dependence of $u_0$ explicitly.
\begin{prop}[Properties of the ground state]\label{SymmetryGS} Consider a path graph with Hamiltonian $H_{\gamma}$ and double-symmetric edge weights, i.e., $\gamma_{vw}=\gamma_{wv}$ and $\gamma_{vw}=\gamma_{-v-w}$ for $v,w \in V_k$. Then $u_0$ is non-degenerate, $u_0(v)$ can be chosen real-valued and positive for all $v \in V_k$. Also $u_0(v)=u_0(-v)$, i.e., the ground state is symmetric. In addition, the ground state is monotonic, i.e., $u_0(n) \leq u_0(n+1)$ for $n \geq 0$.
\end{prop}
\begin{proof} First, since $H_{\gamma}$ is self-adjoint, we may restrict ourselves to real-valued eigenstates. Furthermore, since the ground state is a minimizer of the Rayleigh quotient \cite{Brouwer}, one concludes that all $u_0(v)$ have the same sign (and we choose them to be positive). Positivity then shows that $u_0$ is non-degenerate since two positive eigenvectors cannot be orthogonal to each other. Similarly, the minimization-property of the ground state also directly implies monotonicity. 
	
	Finally, assume that $u_0$ is not symmetric with respect to the zero vertex. Then, you may consider the restrictions of $u_0$ to the vertices $\{-k,...,0\}$ and $\{0,...,k\}$ and extend them in each case to the whole path graph by reflection. Then, for at least one of those vectors, the corresponding Rayleigh quotient  equals the ground-state eigenvalue. However, since the ground state is non-degenerate, this is not possible and hence $u_0$ has to be symmetric. 
 \end{proof}
Recall that the ground-state eigenvalue equals zero for the unweighted path graph without external potential. In the case of non-vanishing external potential, we now aim to show that $\lambda_0(V_k,w,u)$ cannot converge to zero too fast.
\begin{prop}\label{PropGroundStateII} Consider a path graph with Hamiltonian $H_{\gamma}$, $u > 0$ and edge weights $\gamma_{vw}=1$ for $v,w \in V_k$. Then, there exists a constant $c > 0$ such that 
	\begin{equation*}
	c \leq |V_k|^2 \cdot \lambda_0(V_k,w,u)
 	\end{equation*}
 	holds for all $k \in \mathbb{N}$.
\end{prop}
\begin{proof} Assuming the statement does not hold, there exists a subsequence $({k_j})_{j \in \mathbb{N}}$ along which $|V_{k_j}|^2 \cdot \lambda_0(V_{k_j},w,u) \rightarrow 0$. This then implies that 
	\begin{equation*}
	|V_{k_j}|^2 \cdot u_0(0)^2 \rightarrow 0\ 
	\end{equation*}
	along the corresponding subsequence of (normalized) ground states. Consequently, defining $\tilde{u}_0$ to be the vector which agrees with $u_0$ for all $k \neq 0$ and is such that $\tilde{u}_0(0)=0$, one has (again along the subsequence)
	\begin{equation*}
	|V_{k_j}|^2\cdot \mu_0(V_{k_j},w,u) \leq |V_{k_j}|^2\cdot \frac{q_{w,u}[\tilde{u}_0]}{\|\tilde{u}_0\|^2_{\mathcal{H}}}
	\end{equation*}
	where $\mu_0(V_{k_j},w,u)$ denotes the lowest eigenvalues of the unweighted path graph with a Dirichlet vertex at $k=0$. However, for such a graph it is well-known that there exists a constant $\tilde{c} > 0$ with $\tilde{c} \leq |V_k|^2\cdot \mu_0(V_k,w,u)$ for all $k \in \mathbb{N}$. Consequently, one arrives at a contradiction.
\end{proof}
In a next statement we aim to show that the ground state value at zero vertex converges, in the presence of an external potential, faster to zero than in the case of vanishing external potential. Recall that, whenever $u=0$, one has $u_0(0)=\frac{1}{\sqrt{|V_k|}}$. 
\begin{lemma}\label{PropXXX} Consider a path graph with Hamiltonian $H_{\gamma}$, $u > 0$ and edge weights $\gamma_{vw}=1$ for $v,w \in V_k$. Then 
	$$\lim_{k \rightarrow \infty}|V_k|\cdot u_0(0)=0\ . $$
\end{lemma}
\begin{proof} We prove the statement by contradiction and assume that $\lim_{j\rightarrow \infty}|V_{k_j}|\cdot |u_0(0)|=\delta > 0$ for some subsequence $(k_j)_{j \in \mathbb{N}}$.
	
	From the eigenvalue equation and with Proposition~\ref{SymmetryGS} we obtain the relation 
	\begin{equation*}\begin{split}
	2u_0(0)-u_0(-1)-u_0(+1)+u u_0(0)&=2u_0(0)-2u_0(+1)+u u_0(0)\\
	&=\lambda_0(V_k,w,u)u_0(0)
	\end{split}
	\end{equation*}
	which gives
	\begin{equation*}\begin{split}
	|V_k|u_0(+1)&=\left(1+\frac{u-\lambda_0(V_k,w,u)}{2}\right)|V_k|u_0(0) \ .
	\end{split}
	\end{equation*}
	More generally, using the eigenvalue equation we get 
	\begin{equation}\label{RecurrenceRelation}\begin{split}
	u_0(n)=(2-\lambda_0(V_k,w,u))u_0(n-1)-u_0(n-2)
	\end{split}
	\end{equation}
	for all $2 \leq n \leq k-1$. Relation~\eqref{RecurrenceRelation} is a linear recurrence relation and can be solved explicitly. To do this, we consider the two (complex) roots $\lambda_1,\lambda_2$ of the polynomial
	\begin{equation*}
	p(x):=x^2-(2-\lambda_0(V_k,w,u))x+1\ .
	\end{equation*} 
	We obtain, for $|V_k|$ large enough,
	\begin{equation*}
	\lambda_{1,2}=\frac{2-\lambda_0(V_k,w,u)}{2}\pm i \sqrt{1-\left(\frac{2-\lambda_0(V_k,w,u)}{2}\right)^2}\ .
	\end{equation*}
	The solution to \eqref{RecurrenceRelation} is then given by
	\begin{equation*}
	u_0(n)=\alpha_1 \lambda^{n+1}_1+\alpha_2 \lambda^{n+1}_2
	\end{equation*}
	where $\alpha_1, \alpha_2 \in \mathbb{C}$ are determined by the set of equations
	\begin{equation}\label{EqSys}\begin{split}
	u_0(0)&=\alpha_1 \lambda_1 + \alpha_2 \lambda_2\ , \\
	u_0(+1)&=\alpha_1 \lambda^2_1 + \alpha_2 \lambda^2_2 \ .
	\end{split}
	\end{equation}
Multiplying the first equation by $\lambda_2$ and subtracting the second equation yields
\begin{equation}\label{EqAlpha}\begin{split}
\alpha_1=\frac{\lambda_2 u_0(0)-u_0(+1)}{1-\lambda^2_1}\ .
\end{split}
\end{equation}
Also, from the explicit expression for $\lambda_1$ we conclude that there exists a uniform constant $c > 0$ such that, for $k$ large enough,
\begin{equation*}
|1-\lambda^2_1| \leq c \cdot \frac{1}{|V_k|}\ .
\end{equation*}
Hence, since $|V_{k_j}|u_0(0) \rightarrow \delta > 0$, eq.~\ref{EqAlpha} gives
\begin{equation*}\begin{split}
|V_{k_j}||\alpha_1| \geq \tilde{c} \cdot |V_{k_j}|
\end{split}
\end{equation*}
for some uniform constant $\tilde{c} > 0$ and $j \in \mathbb{N}$ large enough. Now, regarding the zero vertex, we also have that 
\begin{equation*}
2|V_{k_j}||\alpha_1| \cdot \cos(\widetilde{\alpha_1}+\widetilde{\lambda}_1) \rightarrow \delta > 0\ ,
\end{equation*}
where $\widetilde{\alpha_1} \in [0,2\pi]$ is the complex angle of $\alpha_1$ and $\widetilde{\lambda}_1 \in [0,2\pi]$ the complex angle of $\lambda_1$. Therefore, since $\lambda_1 \rightarrow 1$, we conclude that $\widetilde{\alpha_1} \rightarrow -\frac{\pi}{2}$. Also, regarding the $(k-p)$-th vertex, $p \in \mathbb{N}$ with $1 \leq p < k-2$, one has
\begin{equation}
|V_k|u_0(k-p)=2|V_k||\alpha_1| \cdot \cos\left(\widetilde{\alpha_1}+(k-p)\widetilde{\lambda}_1\right) \ .
\end{equation}
Now, combining Proposition~\ref{PropPavloI} and Proposition~\ref{PropGroundStateII}, one concludes that there exists a constant $\tilde{l} > 0$ such that, for an arbitrary $\varepsilon > 0$,
\begin{equation*}
0< \tilde{l} \leq (k_j-p)\widetilde{\lambda}_1 \leq \frac{\pi}{2}+\varepsilon
\end{equation*}
holds for all $p \leq \lfloor\sqrt{k_j}\rfloor$ whenever $j$ is large enough. Consequently, there exists a uniform constant $\beta > 0$ such that 
\begin{equation*}
|V_{k_j}|u_0(k_j-p) \geq \beta \cdot |V_{k_j}|
\end{equation*}
holds for all $p \leq \lfloor\sqrt{k_j}\rfloor$ whenever $j$ is large enough. On the other hand, normalization implies 
\begin{equation*}
\sum_{n={k_j}-\lfloor\sqrt{k_j}\rfloor}^{{k_j}-1} |u_0(n)|^2 \leq 1
\end{equation*}
from which we get
\begin{equation*}
\beta^2 \cdot (\lfloor\sqrt{k_j}\rfloor-1)\leq 1\ .
\end{equation*}
This is a contradiction for ${k_j} \in \mathbb{N}$ large enough and hence we conclude the statement.
 \end{proof}
	We are now in position to establish the main theorem of the paper. It is shown that the presence of an external potential indeed leads to a faster decay of the spectral gap in the limit of large volume.
	\begin{theorem}[Convergence of the spectral gap]\label{MainTheoremI} Consider a path graph with Hamiltonian $H_{\gamma}$, $u > 0$ and edge weights $\gamma_{vw}=1$ for $v,w \in V_k$. Then,
		\begin{equation*}
		\lim_{k \rightarrow \infty}|V_k|^2 \cdot \Gamma(V_k,w,u)=0\ .
		\end{equation*}
	\end{theorem}
	\begin{proof} We prove the statement by contradiction and assume that $\lim_{j \rightarrow \infty}|V_{k_j}|^2 \cdot \Gamma(V_{k_j},w,u) > 0$ for a subsequence $(k_j)_{j \in \mathbb{N}}$ which means that $|V_{k_j}|^2\lambda_0(V_{k_j},w,u)$ converges along this subsequence (since $|V_{k_j}|^2\lambda_1(V_{k_j},w,u)$ also converges). 
		
		We start with the basic relation
		\begin{equation*}
		|V_{k_j}|^2\cdot q_{w,u}[u_0]=|V_{k_j}|^2\cdot \lambda_0(V_{k_j},w,u)\ .
		\end{equation*}
		Now, we denote by $\tilde{u}_0 \in \mathcal{H}$ the vector which is obtained from the normalized ground state $u_0 \in \mathcal{H}$ by setting its zeroth component $u_0(0)$ to zero. Employing Proposition~\ref{PropXXX} we conclude
		\begin{equation*}\begin{split}
		\lim_{j \rightarrow \infty}|V_{k_j}|^2\cdot \lambda_0(V_{k_j},w,u)&=\lim_{j \rightarrow \infty}|V_{k_j}|^2\cdot\frac{ q_{w,u}[u_0]}{\|u_0\|^2_{\mathcal{H}}}\ ,\\
		&=\lim_{j \rightarrow \infty}|V_{k_j}|^2\cdot \frac{ q_{w,u}[\tilde{u}_0]}{\|\tilde{u}_0\|^2_{\mathcal{H}}}\ ,
		\end{split}
		\end{equation*}
		noting that all contributions in $q_{w,u}[u_0]$ related to the zero vertex vanish in the limit.
		
		Also, since $\tilde{u}_0$ is zero at the zero vertex, we may restrict $\tilde{u}_0$ to the vertices $\{-k,...,0\}$ and continue it by zero to the vertices $\{0,...,+k\}$. In the same way we consider its restriction to $\{0,...,+k\}$, continued by zero to the vertices $\{-k,...,0\}$. Using that these two restrictions span a two-dimensional subspace of $\mathcal{H}$, the minmax-principle \cite{schmudgen2012unbounded} readily yields
		\begin{equation*}\begin{split}
		\lim_{j \rightarrow \infty}|V_{k_j}|^2\cdot \lambda_1(V_{k_j},w,u)&\leq \lim_{j \rightarrow \infty}|V_{k_j}|^2\cdot\frac{ q_{w,u}[\tilde{u}_0]}{\|\tilde{u}_0\|^2_{\mathcal{H}}}
		\end{split}
		\end{equation*}
		which implies $\lim_{j \rightarrow \infty}|V_{k_j}|^2 \cdot \Gamma(V_{k_j},w,u)= 0$. This leaves us with a contradiction.
	\end{proof}
	\begin{remark} In \cite{KernerTaufer} a conjecture regarding the decay rate of the spectral gap for non-negative and compactly supported potentials was put forward, see also \cite{KernerCompact}. In the context of the present paper and for the operator considered in Theorem~\ref{MainTheoremI}, it again seems plausible that $\Gamma(V_k,w,u)$ cannot close faster than $\sim |V_k|^{-3}$. 
	\end{remark}
	Finally, we establish a result that is directly related to the statement of Theorem~\ref{MainTheoremI}. More explicitly, we shall show that the spectral gap closes also strictly faster than the spectral gap of the unweighted path graph (compare with Proposition~\ref{PropPavloI}) if one is allowing for fast decaying edge weights instead of considering an external potential.
	\begin{theorem}[Upper bound spectral gap]\label{TheoremDecayWeights} Consider a path graph with Hamiltonian $H_{\gamma}$, $u=0$ and double-symmetric edge weights $\gamma_{vw}=\gamma_{wv}$, $\gamma_{vw}=\gamma_{-v-w}$, $v,w \in V_k$, also satisfying the decay condition
		\begin{equation*}
		\gamma_{n(n+1)} \leq \frac{C}{n^\mu}\ , \quad n \in \mathbb{N}\ ,
		\end{equation*}
		for some constants $C > 0$ and $\mu > 1$. Then, there exists $\varepsilon > 0$ and a constant $B > 0$ such that 
		\begin{equation}\label{GZT}
		\Gamma(V_k,w,u) \leq \frac{B}{|V_k|^{2+\varepsilon}}\ 
		\end{equation}
		for all $k$ large enough.
	\end{theorem}
	\begin{proof} We start by recalling that $\lambda_0(V_k,w,u)=0$ for all $k \in \mathbb{N}$. Hence, in order to show~\eqref{GZT} it remains to bound $\lambda_1(V_k,w,u)$. To do this, we employ the minmax-principle and obtain
		\begin{equation*}
		\lambda_1(V_k,w,u) \leq \frac{q_{w,u}[f_k]}{\|f_k\|^2_{\mathcal{H}}}
		\end{equation*}
		where $f_k=(+1,+1,+1,+1,...,0,0,0,...,+1,+1,+1,+1)^T$ has $2\lceil \frac{k}{4}\rceil$ non-zero components. We obtain,
		\begin{equation*}\begin{split}
		\lambda_1(V_k,w,u) &\leq \frac{2C(k-\lceil \frac{k}{4}\rceil)^{-\mu}}{2\lceil \frac{k}{4}\rceil}  \\
		&\leq \frac{\tilde{C}}{k^{1+\mu}}
		\end{split}
		\end{equation*}
		for some $\tilde{C} > 0$ and $k$ large enough. This implies the statement.
	\end{proof}
	\begin{remark} Under the assumptions of Theorem~\ref{TheoremDecayWeights} one has
		\begin{equation*}
		\lim_{k \rightarrow \infty}|V_k|^2\cdot \Gamma(V_k,w,u)=0\ .
		\end{equation*}
		As a consequence, the spectral gap of the weighted Laplacian with edge weights decaying sufficiently fast decays also faster than the spectral gap of the unweighted Laplacian (compare with Proposition~\ref{PropPavloI}). This faster decay of the spectral gap is again a consequence of an effective decoupling of the configuration space in the limit of large volume. In other words, sufficiently fast decaying edge weights lead to a graph that is effectively disconnected in the infinite volume limit. This implies an effective degeneracy of the ground state eigenvalue and this yields a faster decay of the spectral gap.
	\end{remark}

%	\vspace*{0.5cm}
	
	\subsection*{Acknowledgement}{The first author would like to thank D.~Bianchi for interesting discussions. The second author was supported by the project PRIMUS/20/SCI/002 from Charles University.}
	
	\vspace*{0.5cm}
	
	{\small
		\bibliographystyle{amsalpha}
		\bibliography{Literature}}

\end{document}